\documentclass[a4paper]{amsart}

\usepackage[english]{babel}
\usepackage[T1]{fontenc}
\usepackage[utf8]{inputenc}
\usepackage{amssymb}
\usepackage{amsmath}
\usepackage{amsfonts}
\usepackage{dsfont}
\usepackage[numbers]{natbib}
\usepackage{url}

\newcommand{\N}{\mathds{N}}

\newcommand{\R}{\mathds{R}}
\newcommand{\C}{\mathds{C}}

\newcommand{\ergaenzung}[1]{#1}

\providecommand{\abs}[1]{\left\lvert#1\right\rvert}
\providecommand{\norm}[1]{\left\lVert#1\right\rVert}

\DeclareMathOperator{\TextRe}{Re}

\renewcommand{\Re}{\TextRe}

\renewcommand{\exp}{\text{e}}

\newtheorem{lemma}{Lemma}[section]
\newtheorem{prop}[lemma]{Proposition}
\newtheorem{thm}[lemma]{Theorem}
\newtheorem{corollary}[lemma]{Corollary}

\theoremstyle{definition}
\newtheorem{defi}[lemma]{Definition}
\newtheorem{remark}[lemma]{Remark}

\begin{document}

\title[Fractional powers and Dirichlet-to-Neumann operator]{Fractional powers of non-negative operators in Banach spaces via the Dirichlet-to-Neumann operator}
\author[J.~Meichsner]{Jan Meichsner}
\author[C.~Seifert]{Christian Seifert}
\address{TU Hamburg \\ Institut f\"ur Mathematik \\
Am Schwarzenberg-Campus~3 \\
Geb\"aude E \\
21073 Hamburg \\
Germany}
\email{christian.seifert@tuhh.de, jan.meichsner@tuhh.de}



\date{\today}

\begin{abstract}
We consider fractional powers of non-densely defined non-negative operators in Banach spaces defined by means of the Balakrishnan operator.
Under mild assumptions on the operator we show that the fractional powers can partially be obtained 
by a generalised Dirichlet-to-Neumann operator for a Bessel-type differential equation.
\end{abstract}

\maketitle

\section{Introduction}

Fractional powers of linear operators in Banach spaces were studied since the 1950s \cite{ bochner1949, phillips1952, balakrishnan1959} and a huge step was made when A.~V.~Balakrishnan \cite{balakrishnan1960} extended the work from the negatives of generators of bounded semigroups to the wider class of so called non-negative operators. 
In the context of Banach spaces these are the linear operators having their spectrum contained in a sector with vertex $0$ and fulfilling an additional resolvent estimate. 

In 1968 ideas on how to describe fractional powers of the Laplacian via extensions appeared in the context of stochastic processes (see \cite{molchanow1968}) but in this work focus was not on the fractional powers themselves. 

The approach appeared again 40 years later in the context of PDEs in the celebrated work of Cafarelli and Silvestre \cite{caffarelli2007} where the authors described fractional powers of the Laplacian by means of taking traces of functions solving the PDE
\begin{equation}
\begin{aligned}
 \partial_t^2 u(t,x) + \frac{1-2\alpha}{t} \partial_t u(t,x) & =  - \Delta_x u (t,x) \quad & \big( (t,x) \in (0, \infty) \times \R^n \big), \label{fractional_laplace} \\
 u(0,x) & = f(x) \quad & ( x \in \R^n ),
\end{aligned} 
\end{equation}
with $\alpha \in (0,1)$ being the fractional power. 
One can calculate $\left( - \Delta \right)^{\alpha}$ as
\begin{equation}
 c_{\alpha} \left( \left( - \Delta \right)^{\alpha}f \right)(x) = - \lim\limits_{t \rightarrow 0+} t^{1-2\alpha} \partial_t u(t,x) \qquad ( x\in\R^n ), \label{limitdefinition}
\end{equation}
with a constant $c_{\alpha}$ and a solution $u$ of \eqref{fractional_laplace}.  

Formally one could interpret solutions to  \eqref{fractional_laplace} as harmonic functions defined on $\R^{n}\times\R^{2-2\alpha}$. 
In this case the equation \eqref{fractional_laplace} is nothing but the usual Laplacian applied to a function $v$ of the special form 
\begin{align*}
  v: \R^{n}\times\R^{2-2\alpha} \rightarrow \R, \quad  v(x,y) = u \left(\norm{y},x \right),  
\end{align*}
with a suitable function $ u:\R\times \R^{n} \rightarrow \R$. 
So it just depends on the norm of the additional $2-2\alpha$ coordinates. This is the reason why the technique is called harmonic extension. 

The obvious question arises whether this works if one replaces $-\Delta$ in \eqref{fractional_laplace} by a linear operator $A$ acting in a Banach space $X$. 
In this scenario \eqref{fractional_laplace} becomes 
\begin{align}
 u''(t) + \frac{1 - 2 \alpha}{t} u'(t) & = Au(t) \qquad \big( t \in (0, \infty) \big), \label{fractional_ODE} \\
 u(0) & = x, \label{dirichlet}
\end{align}
i.e., a linear ODE in the Banach space $X$ with initial datum $x \in X$. 
If the considered Banach space $X$ is one-dimensional this is just another form of Bessel's differential equation and a functional calculus based on integral representations of its solutions provides solutions to \eqref{fractional_ODE} as well (see \cite{stinga2010, gale2013, arendt2016}). 
The corresponding integral representation of a solution for the case of $A$ being a non-negative selfadjoint second order elliptic differential operator in $L_2(\Omega,\mu)$, where $\Omega\subseteq \R^n$ is open and $\mu$ is a measure on $\Omega$, is due to Stinga and Torrea \cite{stinga2010}.
There, they also establish uniqueness results on the solution $u$ for the case of $A$ having purely discrete spectrum.

Given a solution $u$ to \eqref{fractional_ODE}, one can define analogously to \eqref{limitdefinition}
\begin{equation}
 T_{\alpha}x := - \lim\limits_{t \rightarrow 0+} t^{1-2\alpha} u'(t)   \label{limit}
\end{equation}
(assuming the limit exists) and ask whether 
\begin{equation}
 T_{\alpha} =  c_{\alpha} A^{\alpha} \label{question}
\end{equation}
still holds. 

In case $\alpha = \tfrac{1}{2}$ the limit \eqref{limit} can be interpreted as a normal derivative of $u$ in the domain $[0, \infty)$. 
Hence, the so obtained operator $T_{1/2}$ maps the Dirichlet boundary condition \eqref{dirichlet} on the Neumann boundary condition \eqref{limit} and is therefore called Dirichlet-to-Neumann operator. 
\ergaenzung{The operator, defined as in \eqref{limit}, will turn out to be closable.
So we shall use the terminology Dirichlet-to-Neumann operator for the closure $\overline{T_\alpha}$ of $T_\alpha$ even for $\alpha \neq \tfrac{1}{2}$.}

In \cite{arendt2016,stinga2010} the authors considered the situation when $X$ is a Hilbert space.
In \cite{stinga2010} the equality in \eqref{question} was shown for $L_2$-spaces as noted above for non-negative selfadjoint $A$. Moreover, the constant $c_\alpha$ was explicitly computed.
On the other hand, in \cite{arendt2016} the authors made use of form techniques to study fractional powers. 
In particular, they proved the well-posedness of the Dirichlet problem \eqref{fractional_ODE} for initial data as in \eqref{dirichlet} and showed that the domain of the Dirichlet-to-Neumann operator is a subspace of a complex interpolation space between $X$ and a dense subspace $V$ of it which contains $\mathcal{D}(A)$. 
The considered operator $A$ has bounded inverse and is m-accretive, i.e., non-negative with $M=1$, see Definition \ref{nnops} for more details.  
Such operators have bounded imaginary powers (\cite[Corollary 7.1.8]{haase2006}). 
By \cite[Theorem 11.5.4]{martinez2001} the domains of the fractional powers of these operators coincide with the complex interpolation spaces between $\mathcal{D}(A)$ and $X$ for real powers $\alpha$.
The question whether these interpolation spaces coincide with the domains of the Dirichlet-to-Neumann operators was not completely clarified.  

In \cite{gale2013} the more general situation of $X$ being a Banach space, $-A$ being the generator of a $\beta$ times integrated semigroup, and $\alpha\in\C$ with $0<\Re\alpha<1$ is treated.
Generators of integrated semigroups generalise the notion of semigroup generators. 
Further the ODE \eqref{fractional_ODE} is discussed on an entire sector in $\C$ rather than just on the half line $(0, \infty)$. 
There, \eqref{question} is proved for  $x \in \mathcal{D}(A)$, so
\[
 T_{\alpha} x = c_{\alpha} A^{\alpha} x  \qquad \big( x \in \mathcal{D}(A) \big).
\]
What is missing though is a discussion of the uniqueness of the used extension, i.e, the solution $u$ used to define \eqref{limit} as it was performed in \cite{arendt2016, stinga2010}. 

Note that in general we have $\mathcal{D}(A) \subseteq \mathcal{D}\left( A^{\alpha} \right)$.  
This is where our contribution will come into play. 
Let $A$ be a in general non-densely defined non-negative operator in a Banach space $X$ such that $-A$ generates a bounded semigroup in $\overline{\mathcal{D}(A)}$.  
For $\alpha \in \C$ with $0 < \Re \alpha < 1$ we define fractional powers $A^{\alpha}$ following \cite{martinez2001} by an operator  $J^{\alpha}_A$ associated to $A$ and $\alpha$, its closure $\overline{J^{\alpha}_A}$ which we will refer to as Balakrishnan operator and a suitable extension of it. 
We define the Dirichlet-to-Neumann operator as closure $\overline{T_{\alpha}}$ of the operator $T_{\alpha}$ defined as suggested in \eqref{limit}. 
It will turn out that this closure coincides up to a constant with the Balakrishan operator $\overline{J^{\alpha}_A}$. This is our main result stated in Theorem \ref{maintheorem}.

The paper is organised as follows. In the next section we introduce the basic notion of non-negative operators and define fractional powers. In Section \ref{mainchapter} we define the Dirichlet-to-Neumann operator and proof our main result. The paper ends with a fourth part where we consider an example, namely a multiplication operator in $C_b$. 

At the end of this section let us quickly fix some notation. 
For the remainder of the paper let $X$ be a Banach space and $\alpha \in \C$ with $0< \Re \alpha <1$.

\section{Non-negative operators and fractional powers}

We denote by $L(X)$ the set of all bounded linear operators from $X$ to $X$ and 
\[
\rho(A) := \left\{ \lambda \in \C \mid \text{$(\lambda - A)$ is injective and }\left( \lambda - A \right)^{-1} \in L(X) \right\}
\]
is the resolvent set of a linear operator $A$ in $X$, while $\sigma(A) := \C \setminus \rho(A)$ denotes its spectrum. 
\begin{defi}[non-negative operator]
Let $A$ a be linear operator in $X$. 
Then $A$ is called \emph{non-negative} if $(- \infty, 0) \subseteq \rho\left(A\right)$ and
\[
 M := \sup\limits_{\lambda > 0} \norm{ \lambda \left( \lambda + A \right)^{-1}} < \infty.
\]
\label{nnops}
\end{defi}
We would like to point out that we do not require $0 \in \rho(A)$. 
Non-negative operators with this additional property are usually called positive. 
One can show $M \in [1, \infty)$ (\cite[Corollary 1.1.4]{martinez2001}).
 
For $z \in \C \setminus (-\infty, 0]$ we define $\arg z$ to be the unique number in the interval $(-\pi, \pi)$ such that $z = \abs{z} \exp^{i \arg z}$. 

\begin{defi}[sectorial operator]
Let $A$ be a linear operator in $X$. Then $A$ is called \emph{sectorial} if 
\[
 \exists \theta \in [0, \pi): \sigma(A) \subseteq S_{\theta}:= \left\{ z \in \C \setminus (-\infty, 0]  \mid \abs{\arg z} \leq \theta \right\} \cup \{0\} 
\]
and
\[
 \forall \phi \in (\theta, \pi): \widetilde{M}_{\phi} := \sup\limits_{z \in \C \setminus S_\phi} \norm{z(z-A)^{-1}} < \infty. 
\]
\end{defi}
The minimum of all angles $\theta$ such that $A$ is sectorial in the sense of the above definition is called its \emph{angle of sectoriality} and denoted by $\Theta$. 

By defnition sectoriality implies non-negativity. 
The converse holds true as well, since $ \Theta \leq \pi - \arcsin \left( \tfrac{1}{M} \right)$ (see \cite[Proposition 1.2.1]{martinez2001}).

Let us now come to the Balakrishnan operator and the definition of the fractional powers of $A$. 
Set
\begin{equation}
  \C \setminus (-\infty, 0] \ni z \mapsto z^{\alpha} := \exp ^{\alpha \ln \abs{z} + i \alpha \arg z}. \label{defpower}
\end{equation}
Let $A$ be a non-negative operator in $X$. For $x \in \mathcal{D}(A)$ we define
\begin{equation}
  J^{\alpha}_A x := \frac{\sin\left( \alpha \pi \right)}{\pi} \int\limits_{0}^{\infty} t^{\alpha - 1} \left(  t  + A \right)^{-1} Ax \, dt. \label{bala_integral}
\end{equation}
The integral is convergent in the Bochner sense since
\[
 \int\limits_0^{\infty} \norm{t^{\alpha - 1} \left( A + t \right)^{-1} Ax} \, dt \leq  \int\limits_0^{1} t^{\Re \alpha - 1} \left( M+1 \right) \norm{x} \, dt +  \int\limits_1^{\infty} t^{\Re \alpha - 1} \frac{M}{t} \norm{Ax} \, dt
\]
with $M$ as in Defnition \ref{nnops}. 
\begin{remark}
  The definition of $J_A^\alpha$ can be extended to the more general case $\Re \alpha >0$ but in this paper we shall not be interested in this. 
\end{remark}
 
\begin{remark}
  The basic idea behind this formula is the identity
  \begin{equation}
  \label{ideabala}
    z^{\alpha} = \frac{\sin\left( \alpha \pi \right)}{\pi} \int\limits_{0}^{\infty} t^{\alpha - 1} \frac{z}{t + z} \, dt.
  \end{equation}
  A possible proof for this identity uses a keyhole contour, several limit processes and the fact that the function $t \mapsto t^{\alpha - 1}$ can be extended holomorphically. 
  In contrast to definition \eqref{defpower}, in this situation the domain must be sliced along the positive real axis, i.e., for the integrand $\arg z$ has to be determined in the interval $(0, 2\pi)$ in order to get a holomorphic extension of $t \mapsto t^{\alpha -1}$ inside the chosen contour. 
  Nevertheless, the expression $z^{\alpha}$ is defined as in \eqref{defpower} with $\arg z$ being determined in $(- \pi, \pi)$. 
  In order to make use this formula one interprets $z$ as an operator acting on the Banach space $X = \C$ and extends it by using more general operators; thus arriving at \eqref{bala_integral}.
  
  Using a keyhole contour also establishes the connection between fractional powers defined by means of \eqref{bala_integral} and a holomorphic functional calculus, see Remark \ref{remark:haase_lunardi}. 
 
\end{remark}

The linear operator $J^{\alpha}_A$ is bounded provided $A$ is, it is injective if $A$ is, and it is closable (\cite[Theorem 3.1.8]{martinez2001}). 

\begin{defi}[Balakrishnan operator]
Let $A$ be a non-negative operator in $X$. 
The closure $\overline{J^{\alpha}_A}$ will be called \emph{Balakrishnan operator} with power $\alpha$ and base $A$. 
\end{defi}

The Balakrishnan operator was introduced in \cite{balakrishnan1960} and it is almost the right candidate for the fractional power $A^{\alpha}$.
(A spectral mapping theorem for $\overline{J^{\alpha}_A}$ does not hold in general though, see \cite[Theorem 5.3.1]{martinez2001}.) 
%But if $D \subsetneq X$, a spectral mapping theorem does not hold. 
We now define fractional powers of $A$ with the help of $J^{\alpha}_A$, \ergaenzung{see}  \cite{martinez2001}.  

\begin{defi}[fractional powers]
	Let $A$ be a non-negative operator in $X$. We define the \emph{fractional power $A^{\alpha}$} as 
	\begin{itemize}
		\item[i)] $A^{\alpha} := J^{\alpha}_A$ for $A \in L(X)$ 
		\item[ii)] $A^{\alpha} := \left( J^{\alpha}_{A^{-1}} \right)^{-1}$ for $A$ being unbounded and $0 \in \rho(A)$
		\item[iii)] $A^{\alpha}x := \lim\limits_{\varepsilon \rightarrow 0+} \left( A + \varepsilon \right)^{\alpha}x$ for $A$ being unbounded, $0 \in \sigma(A)$ and $\mathcal{D} \left( A^{\alpha} \right)$ given by 
	\begin{align*}
	\bigl\{ x \in \overline{\mathcal{D}(A)} \mid \exists \, \varepsilon_0 > 0 \, \forall \, 0 < \varepsilon < \varepsilon_0: x \in \mathcal{D}\big( \left(A + \varepsilon \right)^{\alpha} \big),\, \lim\limits_{\varepsilon \rightarrow 0+} \left( A + \varepsilon \right)^{\alpha}x \text{ exists}  \bigr\}.
	\end{align*}
	\end{itemize}
\end{defi}

This yields a well-defined closed linear operator which extends the Balakrishnan operator. 

%As a preparation we need 

\begin{lemma}
	Let $A$ be a non-negative operator in $X$ and $x \in X$. Then
	\[
	x \in \overline{\mathcal{D}(A)} \Leftrightarrow \lim\limits_{n \rightarrow \infty} n (n + A)^{-1}x = x. 
	\]	
	\label{convergence_D}
\end{lemma}

\begin{proof}
	The direction $\Leftarrow$ follows from the fact that $(n (n + A)^{-1}x)_n$ is in $\mathcal{D}(A)$. 
	
	Conversely, consider first $x \in \mathcal{D}(A)$. Then
	\[
	\norm{x - n(n+A)^{-1}x} = \norm{\left( n + A \right)^{-1} Ax} \leq \frac{M \norm{Ax}}{n} \to 0,
	\]	
	so the statement it true for $x \in \mathcal{D}(A)$. Since the sequence $\left( n(n+A)^{-1} \right)_{n \in \N}$ is bounded in $L(X)$ the convergence holds for $x\in \overline{\mathcal{D}(A)}$ by the Banach–Steinhaus Theorem.
\end{proof}

\begin{remark}
	The above lemma can be generalised. 
	Namely \cite[Theorem 6.1.1]{martinez2001}, we have $ \overline{\mathcal{D}(A)} = \overline{\mathcal{D}\left( A^{\alpha} \right)}$ and  for all $x \in X$ 
	\[
	x \in \overline{\mathcal{D}(A)} \Leftrightarrow \lim\limits_{n \rightarrow \infty} \left( n\left(n + A \right)^{-1} \right)^{\alpha}x = x. 
	\]
\end{remark}	

For a linear operator $A$ in $X$ and a closed subspace $D\subseteq X$ we denote by $A_D$ the \emph{part of $A$ in $D$}, i.e., ${\mathcal{D}\left( A_D \right) := \left\{ x \in \mathcal{D}(A)\cap D \mid Ax \in D \right\}}$ and $A_D x :=  Ax$. 
The next proposition clarifies the relationship between $A^{\alpha}$ and $\overline{J^{\alpha}_A}$. 

\begin{prop}
 Let $A$ be a non-negative operator in $X$, $D:=\overline{\mathcal{D}(A)}$. 
 Then
 \begin{itemize}
  \item[i)]  $ \overline{J^{\alpha}_A} = \left( A^{\alpha} \right)_D$, 
  \item[ii)] $\overline{J^{\alpha}_A} = A^{\alpha}$ if and only if $D = X$, that is if and only if $A$ is densely defined, 
  \item[iii)] $\overline{J^{\alpha}_{A_D}} =  \overline{J^{\alpha}_A}$, and therefore $(A_D)^\alpha = \bigl(A^\alpha)_D$. 
 \end{itemize}
\label{bala_power}
\end{prop} 

\begin{proof}
 The proofs for i) and ii) can be found in \cite[Corollary 5.1.12]{martinez2001}. 
 As for the last part we first note that $J^{\alpha}_{A_D} = J^{\alpha}_{A}|_{\mathcal{D}(A_D)}$. Hence $\overline{J^{\alpha}_{A_D}} \subseteq  \overline{J^{\alpha}_A}$. 
 Conversely, for $x \in \mathcal{D} \left(\overline{J^{\alpha}_A} \right)$ choose a sequence $(x_n)$ in $\mathcal{D}(A)$ such that $x_n \rightarrow x$ and $J^{\alpha}_A x_n \rightarrow \overline{J^{\alpha}_A} x$.
 For $n\in\N$ define $y_n := n \left( n + A \right)^{-1} x_n$. Then $(y_n)$ in $\mathcal{D} \left( A^2 \right) \subseteq \mathcal{D}(A_D)$, $y_n \rightarrow x$ and 
 \[
 J^{\alpha}_{A_D} y_n = n \left( n + A \right)^{-1} J^{\alpha}_A x_n \rightarrow \overline{J^{\alpha}_A} x. 
 \] 
 Therefore also $\overline{J^{\alpha}_A} \subseteq \overline{J^{\alpha}_{A_D}}$. The last assertion then follows from i) and ii).
\end{proof}

\begin{remark}
  Besides their use of defining fractional powers as above sectorial operators $A$ can also be used to define linear operators $g(A)$ for suitable holomorphic functions $g$. 
  The function $f$ defined by $f(z) := z^{\alpha}$ is such a function and one can define $f(A)$. 
  By \cite[Proposition 6.2.2]{martinez2001} and \cite[Propositions 3.1.1 and 3.1.12]{haase2006} one has $A^{\alpha} = f(A)$. 

  Furthermore, sectorial operators are generators of analytic semigroups $\left( \exp^{-tA} \right)_{t \geq 0}$ in case $\Theta < \tfrac{\pi}{2}$, which are even strongly continuous if we consider $\left( \left( \exp^{-tA} \right)_D \right)$  
  (see \cite[Proposition 2.1.1 and 2.1.4]{lunardi1995} for details). \label{remark:haase_lunardi}
\end{remark}
  
Let $A$ be a non-negative operator in $X$. As mentioned in the introduction several authors already established 
\[
 c_{\alpha} A^{\alpha} x = \lim\limits_{t \rightarrow 0+} - t^{1-2\alpha}u'(t)
\]
with a solution $u$ of \eqref{fractional_ODE}  with initial datum $x \in \mathcal{D}(A)$. 
Assuming that $u'$ is continuous we have $u'(t) \in \overline{\mathcal{D}(A)}$ for all $t \in (0, \infty)$ and therefore the best we can hope for is
\[
 c_{\alpha} \overline{J^{\alpha}_A}x = \lim\limits_{t \rightarrow 0+} - t^{1-2\alpha}u'(t), \quad u(0) = x \in \mathcal{D}\left( \overline{J^{\alpha}_A} \right), 
\]
since in general $A^{\alpha}x \notin \overline{\mathcal{D}(A)}$.  
	
\section{Dirichlet-to-Neumann operator}
\label{mainchapter}

Throughout this section $A$ is assumed to be a non-negative operator in $X$, $D:=\overline{\mathcal{D}(A)}$, such that $-A_D$ generates a bounded $C_0$-semigroup $\left( T(t) \right)_{t \geq 0}$ in $D$, and let $M:=\sup_{t>0}\norm{T(t)}$. 
Necessarily this means $\Theta \leq \tfrac{\pi}{2}$ while $\Theta < \tfrac{\pi}{2}$ is sufficient. 

We consider the initial value problem 
\begin{equation}
 \begin{split}
  u''(t) + \frac{1 - 2 \alpha}{t} u'(t) & = Au(t) \qquad \big( t \in (0, \infty) \big),  \\
  u(0) & = x, 
 \end{split} \label{ODE}
\end{equation}
with initial datum $x \in D$ be given. 
A function $u$ is considered to be a solution of \eqref{ODE} if ${u \in C_b \bigl( [0, \infty); D \bigr) \cap C^2 \bigl( (0, \infty); D \bigr)}$ such that $u(t) \in \mathcal{D}(A)$ for $t > 0$ and \eqref{ODE} is satisfied.   

\begin{defi}
 For $x \in D$ and $t \in [0, \infty)$ we define  
 \begin{equation}
  U(t)x := \begin{cases} \frac{1}{\Gamma(\alpha)} \left( \frac{t}{2} \right)^{2\alpha} \int\limits_0^\infty r^{-\alpha} \exp^{-\frac{t^2}{4r}} T(r) x \, \frac{dr}{r} & \text{if } t > 0, \\ 
                         x & \text{if } t = 0.
           \end{cases} \label{defi_U}
 \end{equation}
\end{defi}
This definition is \ergaenzung{originally due to} \cite{stinga2010} (for the $L_2$-case), see also \cite{gale2013} for Banach spaces and \cite{arendt2016} for general Hilbert spaces. 
\ergaenzung{Note that we intend to emphasise the interpretation of $U$ as an operator-valued mapping as also performed in \cite{arendt2016}}. 

\ergaenzung{Let us show that $U(\cdot)x$ yields a solution of \eqref{ODE}. 
Although this is contained in \cite[Theorem 2.1]{gale2013} (as well as \cite[Theorem 1.1]{stinga2010} for the $L_2$-case), we will give full proofs for the readers convenience.}

\begin{lemma}
 Let $x \in D$, $u(\cdot):=U(\cdot)x$. Then $ u \in C_b\bigl( [0, \infty); D \bigr) \cap C^{\infty}\bigl( (0, \infty); D \bigr)$. \label{lemma_properties} 
\end{lemma}

\begin{proof}
 Because of 
 \[
  \int\limits_0^{\infty} \norm{r^{-\alpha} \exp^{-\frac{t^2}{4r}} T(r) x} \, \frac{dr}{r} \leq M \norm{x} \int\limits_0^{\infty} r^{-\Re \alpha} \exp^{-\frac{t^2}{4r}} \, \frac{dr}{r} < \infty, 
 \]
 the mapping $u$ is properly defined for $t > 0$. 
 Substituting $s := \tfrac{t^2}{4r}$ yields
 \[
  u(t) = \frac{1}{\Gamma(\alpha)} \int\limits_0^{\infty} s^{\alpha} \exp^{-s} T \left( \tfrac{t^2}{4s} \right) x \, \frac{ds}{s}.
 \]
 Now the boundedness of the semigroup together with its strong continuity and the dominated convergence theorem gives $u \in C_b\bigl( [0, \infty); D \bigr)$. 
 In particular, $u(0)=x$. 
 
 Observe that the integrand as well as the prefactor is smooth for ${t > 0}$. 
 For every such $t$ one may choose appropriate compact intervals $I$ with $t \in I \subset (0, \infty)$ and applies again dominated convergence which shows the smoothness of $u$. 
\end{proof}

\begin{remark}
  By an analogous proof \ergaenzung{one} also \ergaenzung{obtains} $U\in C^\infty\bigl((0,\infty);L(D)\bigr)$.
\end{remark}

\ergaenzung{We can now obtain a solution of \eqref{ODE} for good initial data.}

\begin{prop}
\label{prop:solutionDA}
 For $x \in \mathcal{D}(A)$ the function $u(\cdot) = U(\cdot)x$ is a solution to problem \eqref{ODE}. 
\end{prop}

\begin{proof}
 By linearity we can forget about constant prefactors. For $t> 0$ define
 \[
  g(t) := t^{2 \alpha} \int\limits_0^{\infty} f(t,r) \, dr \quad \text{with} \quad f(t,r):= \exp^{-\frac{t^2}{4r}} r^{-\alpha-1} T(r) x. 
 \]
 One calculates
 \[
  g'(t) = \frac{2\alpha}{t}g(t) + t^{2\alpha}\int\limits_0^{\infty} \left( -\frac{t}{2r} \right) f(t,r) \, dr
 \]
 and
 \[
 \begin{split}
  g''(t) =  - & \, \frac{2\alpha}{t^2}g(t) + \frac{2\alpha}{t}g'(t) + t^{2\alpha}\int\limits_0^{\infty} \left( - \frac{\alpha }{r} \right) f(t,r) \, dr \\
           + \, & t^{2\alpha}\int\limits_0^{\infty} \left( \frac{-1}{2r} \right) f(t,r) \, dr + t^{2\alpha} \int\limits_0^{\infty} \left( \frac{t^2}{4r^2} \right) f(t,r) \, dr. 
 \end{split}
 \]
 Thus, for $t>0$ we obtain
 \[
 \begin{split}
  g''(t) + \frac{1-2\alpha}{t}g'(t) & = t^{2\alpha} \int\limits_0^{\infty} \left( \frac{t^2}{4r^2} - \frac{\alpha+1}{r} \right) f(t,r) \, dr \\
                                    & = t^{2\alpha} \int\limits_0^{\infty} T(r) x \frac{d}{dr} \left( \exp^{-\frac{t^2}{4r}} r^{-\alpha-1} \right) dr \\
                                    & = t^{2\alpha} \int\limits_0^{\infty} \exp^{-\frac{t^2}{4r}} r^{-\alpha-1} A T(r) x \, dr \\
                                    & = Ag(t) 
 \end{split}
 \]
 where we used integration by parts and afterwards Hille's Theorem in the last two steps.    
\end{proof}

The last result extends to $D = \overline{\mathcal{D}(A)}$ \ergaenzung{as was already noted in \cite{gale2013} for the case $D = X$.}  

\begin{prop}
 For all $x \in D$ a solution to problem \eqref{ODE} is given by $u(\cdot) = U(\cdot)x$. 
\end{prop}

\begin{proof}
  Let $x \in D$ be given. Take a sequence $\left( x_n \right)$ in $\mathcal{D}(A)$ with $x_n \rightarrow x$. For $n\in\N$ define $u_n := U(\cdot) x_n$. 
  For $t \in [0, \infty)$ one obtains 
  \[
    \norm{u_n(t) - u(t)} \leq \frac{1}{\abs{\Gamma(\alpha)}} \int\limits_0^{\infty} s^{\Re \alpha-1} \exp^{-s} M \norm{x_n - x} \, ds = \frac{M \norm{x_n - x} \Gamma(\Re \alpha)}{\abs{\Gamma(\alpha)}}. 
  \]
  So $u_n \rightarrow u$ uniformly on $[0, \infty)$. Furthermore, choose a compact interval $[a,b]$ contained in $(0, \infty)$. 
  For $t \in [a,b]$ this results in 
  \begin{align*}
    \norm{u'_n(t) - u'(t)} & \leq C_1 M \int\limits_0^{\infty} r^{- \Re \alpha} \exp^{-\frac{a^2}{4r^2}} \norm{x_n - x} \, \frac{dr}{r} \\
    & +  C_2 M \int\limits_0^{\infty} r^{- \Re \alpha-1} \exp^{-\frac{a^2}{4r^2}} \norm{x_n - x} \, \frac{dr}{r}
  \end{align*}
  with 
  \[
    C_1 : = \max\limits_{t \in [a,b]} \abs{\frac{2\alpha t^{2\alpha-1}}{4^{\alpha} \Gamma(\alpha)}},\quad 
    C_2 := \max\limits_{t \in [a,b]} \abs{ \frac{t^{2\alpha+1}}{2^{2\alpha+1} \Gamma(\alpha)} }.
  \]
  So, $u_n' \rightarrow u'$ uniformly on compact subsets $K \subseteq (0, \infty)$. Similarly one can conclude for higher derivatives. 
  By Proposition \ref{prop:solutionDA}, all $u_n$ fulfil the differential equation \eqref{ODE}. Therefore, for $t \in (0, \infty)$ we get
  \[
    u_n(t) \rightarrow u(t) \quad \text{and} \quad Au_n(t) = u_n''(t) + \frac{1-2\alpha}{t}u'_n(t) \rightarrow u''(t) + \frac{1-2\alpha}{t}u'(t).
  \]
  Since $A$ is closed this yields $u(t) \in \mathcal{D}(A)$ and 
  \[
    Au(t) = u''(t) + \frac{1-2\alpha}{t}u'(t). \qedhere
  \]
\end{proof}

\begin{remark}
  The authors conjecture that all bounded solutions $u$ of \eqref{ODE} are of the form $u( \cdot) = U( \cdot) u(0)$ whenever $-A$ generates a $C_0$-semigroup in $D$. 
  
  This is known in the case $\alpha = \tfrac{1}{2}$. Then the unique bounded solution is given by 
  \begin{equation}
    \label{case1/2}
    u(t) = \exp^{-tA^{\frac{1}{2}}} u(0) \quad(t\geq 0),
  \end{equation}
  see \cite[Theorem 6.3.2]{martinez2001}. 
  In this situation $(\exp^{-tA^{\frac{1}{2}}})_{t\geq 0}$ is a special case of a subordinated semigroup, that is if $-A$ generates a $C_0$-semigroup (in $D$) all fractional powers $-A^{\alpha}$ do as well for $\alpha \in (0,1)$ (\cite[Example 3.4.6]{haase2006}). 
  Even in the case $-A$ is not a generator of a semigroup the unique bounded solution is still given by \eqref{case1/2}.  
  For $\alpha \in (0, \tfrac{1}{2}]$ the operators $-A^{\alpha}$ are generators of holomorphic semigroups \cite[Example 3.4.7]{haase2006}. 
  
  In the Hilbert space case uniqueness results for solutions are also established in \cite{arendt2016,stinga2010}.  
  \label{uniquesolution}
\end{remark}

\ergaenzung{Let us now study the limit in \eqref{limit}. For initial data in $\mathcal{D}(A)$ this limit coincides up to constant with the fractional power of $A$, which was also proved in \cite[Theorem 2.1]{gale2013}.}

\begin{prop}
  \label{prop:DtN=fracpower}
 Let $x \in \mathcal{D}\left( A \right)$. Define $u(\cdot) := U(\cdot)x$ and set
 $
  c_{\alpha} := \frac{\Gamma(1-\alpha)}{2^{2\alpha-1} \Gamma(\alpha)}. 
 $
 Then
 \[
  \lim\limits_{t \rightarrow 0+} - t^{1-2\alpha}u'(t) = c_{\alpha} J^{\alpha}_A x = c_{\alpha} A^{\alpha}x. 
 \]
 \label{same_on_DA}
\end{prop}	

\begin{proof}
  A calculation yields
  \begin{align*}
    -t^{1-2\alpha} u'(t) = \frac{1}{\Gamma(\alpha) 4^{\alpha}}  \Bigl( -2\alpha \int\limits_0^{\infty} r^{-\alpha-1} \exp^{-\frac{t^2}{4r}} T(r) x \, dr 
								      + \frac{t^2}{2} \int\limits_0^{\infty} r^{-\alpha-2} \exp^{-\frac{t^2}{4r}} T(r) x \, dr \Bigr).
  \end{align*} 
  As for the prefactors one observes
  \[
    \frac{-2\alpha}{\Gamma(\alpha)4^{\alpha}} = \frac{\Gamma(1-\alpha)}{2^{2\alpha-1} \Gamma(\alpha) \Gamma(-\alpha)}. 
  \]
  Using this and adding a zero we get
  \begin{align*}
    -t^{1-2\alpha} u'(t) = c_{\alpha} & \left( \frac{1}{\Gamma(-\alpha)} \int\limits_0^{\infty} r^{-\alpha-1} \exp^{-\frac{t^2}{4r}} \left( T(r) x -x \right) dr  \right. \\
				      & - \frac{t^2}{4\alpha \Gamma(-\alpha)} \int\limits_0^{\infty} r^{-\alpha-2} \exp^{-\frac{t^2}{4r}} T(r) x \,dr \\ 
				      & \left. + \frac{1}{\Gamma(-\alpha)} \int\limits_0^{\infty} r^{-\alpha-1} \exp^{-\frac{t^2}{4r}} x \, dr \right). 
  \end{align*} 

  For $x \in \mathcal{D}(A)$ the first integral in the sum exists even for $t = 0$ and yields the desired result 
  \[\lim\limits_{t \rightarrow 0+} - t^{1-2\alpha}u'(t) = c_{\alpha} \frac{1}{\Gamma(-\alpha)} \int\limits_0^{\infty} r^{-\alpha-1} \left( T(r) x -x \right) dr = c_\alpha J^\alpha_A x\]
  by dominated convergence and \cite[Proposition 3.2.1]{martinez2001}.

  So we are left to prove that the sum of the last two integrals converges to $0$ for $t \rightarrow 0+$. 
  Integration by parts gives
  \begin{align*}
    & - \frac{t^2}{4\alpha \Gamma(-\alpha)} \int\limits_0^{\infty} r^{-\alpha-2} \exp^{-\frac{t^2}{4r}} T(r) x \, dr \; + \; 
    \frac{1}{\Gamma(-\alpha)} \int\limits_0^{\infty} r^{-\alpha-1} \exp^{-\frac{t^2}{4r}} x \, dr \\
  = \, &  \frac{t^2}{4\alpha \Gamma(-\alpha)} \int\limits_0^{\infty} r^{-\alpha-2} \exp^{-\frac{t^2}{4r}} \big( x - T(r) x \big) dr. 
  \end{align*}
  Taking the norm of the last expression and using the estimate $\norm{T(r)x - x} \leq \norm{Ax} r$ we obtain
  \[
  \norm{\frac{t^2}{4\alpha \Gamma(-\alpha)} \int\limits_0^{\infty} r^{-\alpha-2} \exp^{-\frac{t^2}{4r}}\big( x - T(r) x \big) dr} 
  \leq \frac{t^{2-2\Re{\alpha}} \norm{Ax}}{4^{1 - \Re \alpha} \abs{\alpha} \abs{\Gamma(-\alpha)}} \int\limits_0^{\infty} s^{\Re{\alpha}-1} \exp^{-s} \, ds,
  \]
  which tends to zero as $t\to 0+$.
\end{proof}

% \begin{defi}
%   We define the \emph{(generalised) Dirichlet-to-Neumann operator} $T_{\alpha}$ as  
%   \begin{align*}
%     \mathcal{D}\left( T_{\alpha} \right) & := \left\{ x \in D \, \big|  \lim\limits_{t \rightarrow 0+} - t^{1-2\alpha}U'(t)x \text{ exists} \right\}, \\
%     T_{\alpha} x &: = \lim\limits_{t \rightarrow 0+} - t^{1-2\alpha}U'(t)x
%   \end{align*}
% \end{defi}

We define the operator $T_{\alpha}$ in $X$ by  
\begin{align*}
  \mathcal{D}\left( T_{\alpha} \right) & := \left\{ x \in D \, \big|  \lim\limits_{t \rightarrow 0+} - t^{1-2\alpha}U'(t)x \text{ exists} \right\}, \\
  T_{\alpha} x &: = \lim\limits_{t \rightarrow 0+} - t^{1-2\alpha}U'(t)x.
\end{align*}

\begin{lemma}
\label{lem:T_alpha_closable}
 The operator $T_{\alpha}$ is closable. 
\end{lemma}	

\begin{proof}
 The proof follows along the lines of proving closability of $J^{\alpha}_A$, see  \cite[Theorem 3.1.8]{martinez2001}. 
 The crucial ingredient is the fact that one has $A^{\alpha} (\lambda + A)^{-1} \in L(X)$ for any $\lambda > 0$ (\cite[Theorem 3.1.8]{martinez2001}). 
 So, let $\left( x_n \right)$ in $\mathcal{D}\left(T_\alpha \right)$, $x_n \rightarrow 0$, $T_{\alpha} x_n \rightarrow y$.  
 Then, for $n\in\N$ we have $(\lambda+A)^{-1}x_n\in\mathcal{D}(A)$ and by Proposition \ref{prop:DtN=fracpower} we obtain
 \[\lim\limits_{t \rightarrow 0+} -t^{1 - 2 \alpha} U'(t) ( \lambda + A)^{-1} x_n = c_{\alpha} J^{\alpha}_A (\lambda + A)^{-1} x_n.\]
 Thus
 \begin{align*}
  (\lambda + A)^{-1} y & = \lim\limits_{n \rightarrow \infty} \lim\limits_{t \rightarrow 0+} -t^{1 - 2 \alpha} ( \lambda + A)^{-1} U'(t)x_n \\
  & = \lim\limits_{n \rightarrow \infty} \lim\limits_{t \rightarrow 0+} -t^{1 - 2 \alpha} U'(t) ( \lambda + A)^{-1} x_n
  =  \lim\limits_{n \rightarrow \infty} c_{\alpha} J^{\alpha}_A (\lambda + A)^{-1} x_n = 0.
 \end{align*}
 Hence, $y=0$.
\end{proof}

\ergaenzung{Lemma \ref{lem:T_alpha_closable} gives rise to the following definition.}

\begin{defi}
  We call $\overline{T_\alpha}$ the \emph{(generalised) Dirichlet-to-Neumann operator}.
\end{defi}

\begin{thm}
\label{maintheorem}
  We have $c_{\alpha} \overline{J^{\alpha}_A} = \overline{T_{\alpha}}$.
\end{thm}

\begin{proof}
We already know $ c_{\alpha} J^{\alpha}_A = T_{\alpha}\big|_{\mathcal{D}(A)} \subseteq T_{\alpha}$. 	
Hence, $ c_{\alpha}\overline{J^{\alpha}_A} \subseteq \overline{T_{\alpha}} $ follows. 

For the other inclusion consider first $x \in \mathcal{D}(T_{\alpha})$ and $\lambda >0$. 
Then
\begin{align*}
 \left( \lambda + A \right)^{-1} T_{\alpha}x & = \lim\limits_{t \rightarrow 0+} -t^{1 - 2 \alpha} (\lambda + A)^{-1} U'(t)x \\
 & = T_{\alpha} \left( \lambda + A \right)^{-1} x = c_\alpha J^\alpha_A \left( \lambda + A \right)^{-1} x.
\end{align*}
Now take $x \in \mathcal{D}\left( \overline{T_{\alpha}} \right)$ and a sequence $(x_n)$ in $\mathcal{D}(T_{\alpha})$ with $x_n \rightarrow x$ such that $T_{\alpha} x_n \rightarrow \overline{T_{\alpha}} x$. 
Observe that we have $\left(T_{\alpha} x_n \right)$ in $D$ and apply Lemma \ref{convergence_D}
together with the preliminary result which yields 
\[
 \lim\limits_{n \rightarrow \infty} n(n + A)^{-1} T_{\alpha} x_n = \overline{T_{\alpha}} x = \lim\limits_{n \rightarrow \infty} c_{\alpha} J_A^{\alpha} \left( n(n + A)^{-1} x_n \right).
\]
The sequence $\left( n(n + A)^{-1} x_n \right)$ is contained in $\mathcal{D}(A) = \mathcal{D}(J^{\alpha}_A)$ and converges to $x$. By the closability of $J^{\alpha}_A$ we get $x \in \mathcal{D}\left( \overline{J^{\alpha}_A} \right)$ and 
\[
 \overline{T_{\alpha}} x =  c_{\alpha} \overline{J^{\alpha}_A} x. \qedhere
\]  

\end{proof}

\begin{remark}
  \textrm{(a)}
	If $\alpha = \tfrac{1}{2}$ we have $T_{1/2} = \overline{T_{1/2}}$ by Remark \ref{uniquesolution}, i.e.\ $T_{1/2}$ is closed.
	This follows from the fact that if $x \in \mathcal{D} \bigl( A^{\frac{1}{2}} \bigr)$ we have 
	\[
	 \lim\limits_{t \rightarrow 0+} - U'(t)x = \lim\limits_{t \rightarrow 0+} A^{\frac{1}{2}} \exp^{-tA^{\frac{1}{2}}} x = A^{\frac{1}{2}}x. 
	\]
	Thus, $x \in \mathcal{D} \left( T_{1/2} \right)$ and $T_{1/2} x = A^{\frac{1}{2}} x $. 
	
  \textrm{(b)}
    We conjecture that $T_\alpha$ is always closed.
\end{remark}

Combining Theorem \ref{maintheorem} with Proposition \ref{bala_power} we obtain the following.

\begin{corollary}
 Let $A$ be a non-negative operator in $X$, $D:=\overline{\mathcal{D}(A)}$. Let $-A_D$ generate a bounded strongly continuous semigroup on 
 $D$, $\alpha \in \C$ with ${0< \Re \alpha <1}$ and denote by $T_{\alpha}$ the associated Dirichlet-to-Neumann operator. 
 Then we have ${\overline{T_{\alpha}}  = c_{\alpha} \overline{J^{\alpha}_A} \subseteq c_{\alpha} A^{\alpha}}$ and 
 $\overline{T_{\alpha}} = c_{\alpha} A^{\alpha}$ if and only if $A$ is densely defined.  
\end{corollary}

\section{Example}

We shall demonstrate the above considerations at the rather easy example of a multiplication operator on $C_b(\Omega)$ for some open set $\Omega \subseteq \R^n$. 

Let $f \in C(\Omega)$ with $\overline{\mathcal{R}(f)} \subseteq S_{\theta}$ for some $\theta \in \left[0, \tfrac{\pi}{2} \right]$ and define 
\[
 \mathcal{D} \left( A \right) := \{g \in C_b(\Omega)  \mid fg \in C_b(\Omega) \}, \quad Ag := fg. 
\]  
The operator $A$ is bounded if and only if $f$ is bounded, it is closed and in general not densely defined. 
It is non-negative and $-A$ generates the semigroup $\left(T(t) \right)_{t \geq 0}$ given by
\[
 \left( T(t) g \right) (x) = \text{e}^{-tf(x)} g(x) \qquad \big( x\in\Omega, t\geq 0, g\in C_b(\Omega) \big), 
\]
which is strongly continuous on $\overline{\mathcal{D}(A)}$ and analytic for $\theta < \tfrac{\pi}{2}$. 
We claim 
\[
 \mathcal{D} \left( A^{\alpha} \right) = \{ g \in C_b(\Omega) \mid f^{\alpha} g \in C_b(\Omega) \}. 
\]
So let $g \in \mathcal{D} \left( A^{\alpha} \right)$. Then
\begin{equation}
 \left( A^{\alpha} g \right)(x) = \lim\limits_{\varepsilon \rightarrow 0+}  \left( \left( J^{\alpha}_{\left( A + \varepsilon \right)^{-1}} \right)^{-1} g \right) (x) \label{maxdomain}
\end{equation}
exists. One evaluates this expression using formula \eqref{ideabala} and the equality
\[
 z^{-\alpha} = \left( z^{\alpha} \right)^{-1} = \left( z^{-1} \right)^{\alpha} =  \frac{\sin\left( \alpha \pi \right)}{\pi} \int\limits_{0}^{\infty} t^{- \alpha } \frac{z}{t + z} \, dt. 
\]
The result is 
\[
 \left( A^{\alpha} g \right) (x) = f(x)^{\alpha} g(x) \quad(x\in\Omega),
\]
and thus $f^{\alpha} g \in C_b (\Omega)$. 

Conversely we make use of formula \eqref{ideabala} to estimate the term $\abs{\left(f(x) + \varepsilon \right)^{\alpha} - f(x)^{\alpha}}$ for $\varepsilon > 0$ given.  
Denote 
\[
 M := \sup\limits_{x \in \R^n, \, t>0} \abs{\frac{t}{f(x) + t} }
\]
which is the non-negativity constant of $A$. Then
\[
\sup\limits_{x \in \R^n, \, t>0} \abs{\frac{f(x)}{f(x) + t}} = \sup\limits_{x \in \R^n, \, t>0} \abs{-\frac{t}{f(x) + t} + 1}\leq M+1.
\]
Now we can estimate
\[
 \int\limits_0^{\varepsilon} t^{\Re \alpha-1} \abs{\frac{f(x)+\varepsilon}{f(x)+\varepsilon + t} - \frac{f(x)}{f(x) + t}} \, dt \leq 2(M+1) \frac{\varepsilon^{\Re \alpha}}{\Re \alpha}
\]
and
\begin{align*}
 & \int\limits_{\varepsilon}^{\infty} t^{\Re \alpha-1} \abs{\frac{f(x)+\varepsilon}{f(x)+\varepsilon + t} - \frac{f(x)}{f(x) + t}} \, dt \\
 \leq \; &  \int\limits_{\varepsilon}^{\infty} t^{\Re \alpha} \abs{\frac{\varepsilon}{\left( f(x)+\varepsilon + t \right) \left( f(x) + t \right)} } \, dt \\ 
 \leq \; & M^2 \frac{\varepsilon^{\Re \alpha}}{1 - \Re \alpha}. 
\end{align*}
Thus, there is a positive constant $C$ such that 
\[
 \norm{ \left( f + \varepsilon \right)^{\alpha} - f^{\alpha}} \leq C \varepsilon ^{\Re \alpha}. 
\]
In particular, $\left( f + \varepsilon \right)^{\alpha} - f^{\alpha} \in C_b(\Omega)$ and therefore also $\left( f + \varepsilon \right)^{\alpha} g \in C_b(\Omega)$ since $f^{\alpha} g \in C_b(\Omega)$ by assumption. Furthermore $\left( f + \varepsilon \right)^{\alpha} g \rightarrow f^{\alpha} g$ which implies $g \in \mathcal{D} \left( A^{\alpha} \right)$. 

\begin{remark}
  The presented proof resembles the abstract fact that $\left( A + \varepsilon \right)^{\alpha} - A^{\alpha}$ can be extended to a bounded operator in $D$, see also \cite[Proposition 5.1.14]{martinez2001}. 
\end{remark}

Applying Proposition \ref{bala_power} we obtain $ \mathcal{D} \left( \overline{J^{\alpha}_A} \right) = \{ g \in \overline{\mathcal{D}(A)} \, \big| \, f^{\alpha} g \in \overline{\mathcal{D}(A)} \}. $

Furthermore, for $g \in C_b(\Omega)$, $t > 0$, $x \in \Omega$ we can evaluate the expression $\left(U(t) g \right)(x)$ from (\ref{defi_U}) explicitly using \cite[10.32.10]{nist2016}, which results in

\begin{align*}
 u(t,x) = \big( U(t)g \big)(x) = & \;  \frac{1}{\Gamma(\alpha)} \left( \frac{t}{2} \right)^{2\alpha} \int\limits_0^\infty r^{-\alpha} \exp^{-\frac{t^2}{4r}} e^{-rf(x)} g(x) \, \frac{dr}{r} \\
 = & \; \frac{2 g(x)}{\Gamma(\alpha)} \left( \frac{t f(x)^{\frac{1}{2}}}{2} \right)^{\alpha} K_{\alpha}\left( t f(x)^{\frac{1}{2}} \right) 
\end{align*}
with the modified Bessel function $K_{\alpha}$ (see \cite[10.25.2 and 10.27.4]{nist2016} for a definition).  
By direct calculation one verifies now that
\[
 u(t,x) =  \frac{2 g(x)}{\Gamma(\alpha)} \Bigl( \frac{t f(x)^{\frac{1}{2}}}{2} \Bigr)^{\alpha} \Bigl( \frac{\pi}{2 \sin (\alpha \pi)} \frac{2^{\alpha} t^{-\alpha} f(x)^{-\frac{\alpha}{2}}}{\Gamma(1 - \alpha)} + \mathcal{O}(t^{\alpha}) \Bigr) \rightarrow g(x)\quad(t\to 0+). 
\]
A second calculation based on the same ideas as before shows
\begin{equation*}
\begin{aligned}
 \frac{\partial_t u(t,x)}{t^{2\alpha - 1} } & = \frac{f(x)^{\frac{\alpha}{2}} g(x)}{\Gamma(\alpha) 2^{\alpha - 1}} 
 \Biggl(  \frac{\alpha \pi }{2 \sin(\alpha \pi)} \Bigl( \frac{\bigl(t^{2} f(x)^{\frac{1}{2}} \bigr)^{-\alpha} }{2^{-\alpha} \Gamma(1-\alpha)} + \frac{t^{1-2\alpha} f(x)^{\frac{1-\alpha}{2}}}{2^{1-\alpha} \Gamma(2-\alpha)}  - \frac{ f(x)^{\frac{\alpha}{2}}}{2^{\alpha} \Gamma(1+\alpha)} \Bigr) \\ 
 & \hspace{2,3cm} + \frac{\pi }{2 \sin(\alpha \pi)} \Bigl(   \frac{\bigl(t^{2} f(x)^{\frac{1}{2}} \bigr)^{-\alpha} }{2^{-\alpha} \Gamma(-\alpha)} + \frac{t^{1-2\alpha} f(x)^{\frac{1-\alpha}{2}}}{2^{1-\alpha} \Gamma(1-\alpha)}  - \frac{ f(x)^{\frac{\alpha}{2}}}{2^{\alpha} \Gamma(\alpha)}  \Bigr) \\ 
 & \hspace{2,3cm} +  \mathcal{O} \bigl(t^{2-2\alpha} \bigr) \Biggr) \\
 & = f(x)^{\frac{\alpha}{2}} g(x) \Bigl( \frac{\Gamma(1-\alpha)}{2^{\alpha}} \Bigr) \Bigl( -\frac{2 f(x)^{\frac{\alpha}{2}}}{2^{\alpha} \Gamma(\alpha)}  + \mathcal{O} \bigl( t^{2-2\alpha} \bigr) \Bigr). 
 \end{aligned}
 \end{equation*}
From this, and as expected, we derive
\[
 \lim\limits_{t \rightarrow 0+} - t^{1-2\alpha} \partial_t u(t,x) = c_{\alpha} f(x)^{\alpha} g(x). 
\]
The calculated limits are just pointwise limits though. In order to check uniform convergence one would have to use properties of $K_{\alpha}$ different from its power series representation.  

\bibliographystyle{plainnat}

\bibliography{paper_dtn}

\begin{thebibliography}{13}
\providecommand{\natexlab}[1]{#1}
\providecommand{\url}[1]{\texttt{#1}}
\expandafter\ifx\csname urlstyle\endcsname\relax
  \providecommand{\doi}[1]{doi: #1}\else
  \providecommand{\doi}{doi: \begingroup \urlstyle{rm}\Url}\fi

\bibitem[nis()]{nist2016}
{NIST Digital Library of Mathematical Functions}.
\newblock Version 1.0.14, visited 10.03.2017 \url{http://dlmf.nist.gov/}.

\bibitem[Arendt et~al.()Arendt, ter Elst, and Warma]{arendt2016}
W.~Arendt, A.~F.~M. ter Elst, and M.~Warma.
\newblock {Fractional powers of sectorial operators via the
  Dirichlet-to-Neumann operator}.
\newblock arxiv preprint \url{https://arxiv.org/abs/1608.05707}.

\bibitem[Balakrishnan(1959)]{balakrishnan1959}
A.~V. Balakrishnan.
\newblock {An operational calculus for infinitesimal generators of semigroups}.
\newblock \emph{Trans. Amer. Math. Soc.}, 91:\penalty0 330--353, 1959.

\bibitem[Balakrishnan(1960)]{balakrishnan1960}
A.~V. Balakrishnan.
\newblock Fractional powers of closed operators and the semigroups generated by
  them.
\newblock \emph{Pacific J. Math.}, 10\penalty0 (2):\penalty0 419--437, 1960.

\bibitem[Bochner(1949)]{bochner1949}
S.~Bochner.
\newblock {Diffusion Equation and Stochastic Processes}.
\newblock \emph{Proc. Nat. Acad. Sciences}, 35:\penalty0 368--370, 1949.

\bibitem[Caffarelli and Silvestre(2007)]{caffarelli2007}
L.~Caffarelli and L.~Silvestre.
\newblock An extension problem related to the fractional laplacian.
\newblock \emph{Comm. Partial Differential Equations}, 32\penalty0
  (8):\penalty0 1245--1260, 2007.

\bibitem[Gal{\'e} et~al.(2013)Gal{\'e}, Miana, and Stinga]{gale2013}
J.~E. Gal{\'e}, P.~J. Miana, and P.~R. Stinga.
\newblock Extension problem and fractional operators: semigroups and wave
  equations.
\newblock \emph{Journal of Evolution Equations}, 13\penalty0 (2):\penalty0
  343--368, 2013.

\bibitem[Haase(2006)]{haase2006}
M.~Haase.
\newblock \emph{The Functional Calculus for Sectorial Operators}.
\newblock Operator Theory: Advances and Applications. Birkh{\"a}user Basel,
  2006.

\bibitem[Lunardi(1995)]{lunardi1995}
A.~Lunardi.
\newblock \emph{Analytic Semigroups and Optimal Regularity in Parabolic
  Problems}.
\newblock Modern Birkh{\"a}user Classics. Springer Basel, 1995.

\bibitem[Martinez and Sanz(2001)]{martinez2001}
C.~Martinez and M.~Sanz.
\newblock \emph{The Theory of Fractional Powers of Operators}.
\newblock North-Holland Mathematics Studies. Elsevier Science, 2001.

\bibitem[Molchanow and Ostrovskii(1968)]{molchanow1968}
S.~A. Molchanow and E.~Ostrovskii.
\newblock Symmetric stable processes as traces of degenerate diffusion
  processes.
\newblock \emph{Theory Probab. Appl.}, 14\penalty0 (1):\penalty0 128--131,
  1968.

\bibitem[Phillips(1952)]{phillips1952}
R.~S. Phillips.
\newblock {On the generation of semigroups of linear operators}.
\newblock \emph{Pacific J. Math.}, 2:\penalty0 343--369, 1952.

\bibitem[Stinga and Torrea(2010)]{stinga2010}
P.~R. Stinga and J.~L. Torrea.
\newblock {Extension Problem and Harnack's Inequality for Some Fractional
  Operators}.
\newblock \emph{Comm. Partial Differential Equations}, 35\penalty0
  (11):\penalty0 2092--2122, 2010.

\end{thebibliography}

\end{document}